\documentclass[a4paper,11pt,leqno]{amsart}
\usepackage{a4wide}
\usepackage{amsmath}
\usepackage[italian, english]{babel}
\usepackage{amsfonts}
\usepackage{amssymb}
\usepackage{dsfont}
\usepackage{mathrsfs}
\usepackage{graphicx}
\usepackage{fancyhdr}
\usepackage{multicol}
\usepackage{enumitem}
\usepackage{amsthm}
\usepackage{empheq}
\usepackage{cases}
\usepackage[all]{xy}
\usepackage{stmaryrd}
\usepackage[colorlinks, citecolor=blue, linkcolor=red]{hyperref}
\usepackage{mathrsfs}
\usepackage{tikz,tikz-cd}
\usepackage{cancel}
\def\SS{{{\mathbb S}}}

\def\RR{{\mathbb R}}

\tikzset{
	subset/.style={
		draw=none,
		edge node={node [sloped, allow upside down, auto=false]{$\subset$}}},
	Subset/.style={
		draw=none,
		every to/.append style={
			edge node={node [sloped, allow upside down, auto=false]{$\subset$}}}
	}
}
\tikzset{
	labl/.style={anchor=south, rotate=90, inner sep=.50mm}
}

\newcommand{\ci}{\mathds{C}}

\newcommand{\riem}{\operatorname{Riem}}
\newcommand{\ricc}{\operatorname{Ric}}
\newcommand{\weyl}{\operatorname{W}}
\newcommand{\diver}{\operatorname{div}}

\newcommand{\cp}{\ci\mathbb{P}}

\newcommand\restrict[1]{\raisebox{-.5ex}{$|$}_{#1}}
\newcommand{\bigslant}[2]{{\raisebox{.0em}{$#1$}\left/\raisebox{-.0em}{$#2$}\right.}}
\newcommand{\KN}{\mathbin{\bigcirc\mspace{-15mu}\wedge\mspace{3mu}}}

\newcommand{\longra}{\longrightarrow}

\newcommand{\set}[1]{{\left\{#1\right\}}}               
\newcommand{\pa}[1]{{\left(#1\right)}}                  
\newcommand{\sq}[1]{{\left[#1\right]}}                  
\newcommand{\abs}[1]{{\left|#1\right|}}                 
\newcommand{\pair}[1]{\left\langle#1\right\rangle}      

\newcommand{\metric}{\pair{\;, }}                          

\newcommand{\ol}[1]{\overline{#1}}

\renewcommand{\tilde}[1]{\widetilde{#1}}

\newcommand{\qd}[1]{\mathbin{{Q}_{#1}\mspace{-25mu}\tiny\raisebox{0.65ex}{$2$}\mspace{20mu}}}
\newcommand{\qt}[1]{\mathbin{{Q}_{#1}\mspace{-24.7mu}\tiny\raisebox{0.65ex}{$3$}\mspace{20mu}}}
\newcommand{\qq}[1]{\mathbin{{Q}_{#1}\mspace{-25.6mu}\tiny\raisebox{0.7ex}{$4$}\mspace{20mu}}}

\newcommand{\pmat}[1]{{\begin{pmatrix}#1\end{pmatrix}}} 

\newtheorem{theorem}{\textbf{Theorem}}[section]

\newtheorem{proposition}[theorem]{\textbf{Proposition}}
\newtheorem{cor}[theorem]{\textbf{Corollary}}

\theoremstyle{remark}
\newtheorem{rem}[theorem]{\textbf{Remark}}

\numberwithin{equation}{section}
\title[Einstein metrics and twistor spaces]
{A note on Einstein metrics and Riemannian twistor spaces}

\date{\today}


\begin{document}
	\maketitle
	
	\begin{center}
		\textsc{\textmd{D. Dameno \footnote{Dipartimento di Matematica "Giuseppe Peano", Universit\`{a} degli Studi di Torino, Via Carlo Alberto 10, 10123 Torino, Italy.
					Email: davide.dameno@unito.it.}}}
	\end{center}

	\begin{abstract} 
		Inspired by the problem of classifying Einstein manifolds
		with positive scalar curvature, we prove that an Einstein
		four-manifold whose associated twistor space
		has scalar curvature constant on the fibers of the twistor bundle
		is half conformally flat: in particular, the only compact 
		Einstein four-manifolds with positive scalar curvature 
		satisfying this twistorial condition are $\SS^4$ and $\cp^2$. We 
		also generalize a well-known result due to Friedrich and 
		Grunewald, providing a classification of complete four-manifolds 
		whose twistor space is Ricci parallel. 
	\end{abstract}

	\section{Introduction}
	
	Given a smooth manifold $M$, we say that a metric $g$ on $M$ is
	\emph{Einstein} if the Ricci tensor $\ricc$ is constant,
	i.e. if there exists $\lambda\in\RR$ such that $\ricc=\lambda g$.
	These metrics, which naturally arise in the context of General Relativity and instantons theory, 
	acquire additional importance in the four-dimensional
	case: first, in lower
	dimensions, the Einstein condition is completely understood,
	since it implies that the sectional curvature is constant,
	while in higher dimensions it seems to be less rigid (see, e.g., 
	\cite{boyer}).
	Furthermore, four-dimensional Riemannian manifolds 
	carry unique geometric features and the existence of Einstein
	metrics implies remarkable topological conclusions
	(see e.g. \cite{hitch, lebrunhit, thorpe}).
	
	In this paper, we are interested in
	Einstein four-manifolds with non-negative scalar curvature, whose 
	classification has been addressed by many authors: for instance,
	LeBrun classified these spaces in the compact case 
	requiring the existence of a 
	compatible complex or symplectic structure
	(\cite{lebrunsymp}), while Wu characterized positive
	Einstein metrics on simply connected four-manifolds, under the
	condition $\det\weyl^+\geq 0$, where $\weyl^+$ is the
	self-dual Weyl tensor (\cite{wudet}, then generalized in \cite{lebrun}),
	exploiting techniques derived in a seminal paper by Derdzi\'{n}ski
	(\cite{derd}). A beautiful rigidity result
	was obtained by Gursky (\cite{gursky}), by means of a sharp integral 
	inequality involving the Weyl tensor, later extended by the author,
	Branca, Catino and Mastrolia (\cite{bcdmpinch}).
	
	In this paper, we provide a new characterization of Einstein
	four-manifolds by means of Riemannian twistor theory: inspired
	by the well-known Penrose's programme, twistor spaces in the
	Riemannian context were introduced by Atiyah, Hitchin and Singer
	in their studies of self-dual solutions for Yang-Mills equations
	(\cite{athisin}). 
	The aim of 
	this approach was to interpret the conformal information of $(M,g)$
	in terms of Complex Geometry: indeed, Riemannian twistor spaces are
	\emph{conformally invariant} (\cite{athisin, debnan}) and they 
	admit a natural almost complex structure $J$ 
	which happens to be integrable
	if and only if the (anti-)self-dual Weyl tensor of $(M,g)$ identically
	vanishes (\cite{athisin}). This almost complex structure was 
	later studied by many authors: one of
	the most remarkable results is due to Hitchin (\cite{hitch2}), 
	who showed that, if $(M,g)$ is compact and 
	its associated twistor space admits a K\"{a}hler structure with
	respect to $J$, then $(M,g)$ is either $\cp^3$ or the flag manifold 
	$F_{1,2}$, respectively
	associated to $\SS^4$ and $\cp^2$ (for a more detailed
	discussion about classical results in this direction, 
	we refer the reader to the
	surveys \cite{davmusgran, lebruntour}). 
	
	Twistor spaces
	also admit a Riemannian structure, compatible with the complex one, 
	given by a family $\{g_t\}_{t>0}$ 
	of Riemannian metrics with respect to which 
	the twistor bundle projection
	becomes a Riemannian submersion with totally geodesic fibers 
	(\cite{besse}). 
	Many classification results were obtained analysing the 
	Riemannian/Hermitian structure of the twistor spaces (see e.g.
	\cite{davmussurv} and the references therein): however, very few
	results are known if one wants to consider Einstein four-manifolds
	which are not \emph{half conformally flat} (\cite{cdmnachr, reznikov}). 
	This is due to
	the fact that, in this case, the twistor space is only an almost Hermitian
	manifold: furthermore, the Einstein condition is not conformally 
	invariant. Nevertheless, here we provide a partial result 
	towards the classification of Einstein four-manifolds with
	positive scalar curvature, assuming a curvature condition
	on the associated twistor space: namely, we prove the following
	
	\begin{theorem} \label{constscaltwist}
		Let $(M,g)$ be an Einstein 
		four-manifold and $(Z,g_t)$ be its twistor space. If the 
		scalar curvature $\ol{S}$ of $(Z,g_t)$ is
		constant on the fibers of the twistor bundle, then 
		$(M,g)$ is self-dual.
	\end{theorem}
	Now, we recall that Friedrich and Kurke proved that the 
	twistor space $(Z,g_t,J)$ of an oriented Riemannian four-manifold
	$(M,g)$ is K\"{a}hler-Einstein
	if and only if $(M,g)$ is Einstein and
	self-dual with positive scalar curvature equal to $12/t^2$
	\cite{frikur} (see also \cite{mus} and 
	\cite[Theorem 4.5]{cdmnachr} for some generalizations): 
	combining this with the aforementioned result by Hitchin, we are
	able to show the following
	\begin{cor} \label{corollaryeinst}
		Under the hypotheses of Theorem \ref{constscaltwist}, if 
		$(M,g)$ has positive scalar curvature $S$, then, up to a constant
		rescaling of $g$, 
		its twistor space $(Z,g_t,J)$ is a K\"{a}hler-Einstein manifold.
		In particular, if $M$ is compact, then $(M,g)$ is homotetically 
		isometric to 
		$\SS^4$ or $\cp^2$, both endowed with their standard metrics. 
	\end{cor}
	Theorem \ref{constscaltwist} 
	is obtained by explicit computations involving the curvature
	of the twistor space $Z$ (derived by the methods described in 
	\cite{jenrig} and already appearing in \cite{cdmagag, cdmnachr, 
	damenothesis}), the
	peculiar decomposition of the Riemann tensor in dimension four 
	and a well-known argument due to
	Derdzi\'{n}ski (\cite{derd}), translating the
	hypothesis on the scalar curvature $\ol{S}$ into a condition
	regarding the anti-self-dual Weyl tensor $\weyl^-$.  
	
	On the other hand, we deal with a generalization of the Einstein condition
	on the twistorial level: we recall that Friedrich and Grunewald
	classified Einstein twistor spaces $(Z,g_t)$, showing
	that this case occurs if and only if $(Z,g_t)$ is either
	$\cp^3$ or $F_{1,2}$, endowed with their standard K\"{a}hler-Einstein 
	structures or with their strictly nearly K\"{a}hler structures
	(\cite{frigru}). Here we analyse the \emph{Ricci parallel} condition,
	i.e. $\nabla\ricc\equiv 0$, which is satisfied by all Einstein
	manifolds (note that a classification result for
	Ricci parallel twistor spaces was obtained in 
	\cite{cdmagag}, under an extra assumption): exploiting the well-known
	fact that complete Ricci parallel manifolds split isometrically 
	as products of Einstein manifolds, we are able to obtain the 
	following
	
	\begin{theorem} \label{twistricpar}
		Let $(M,g)$ be a complete Riemannian four-manifold and 
		$(Z,g_t)$ its twistor space. Then, $(Z,g_t)$ is Ricci parallel
		if and only if $(Z,g_t)$ is Einstein or $(M,g)$ is self-dual
		and Ricci-flat. 
	\end{theorem}
	
	We provide two proofs of this result: one is obtained \emph{via} explicit
	computations of the Levi-Civita connection forms and the
	de Rham decomposition Theorem, while the other is based
	on a characterization of the integrability of the horizontal
	distribution of the Riemannian submersion given by the
	bundle projection. It is well-known that self-dual, Ricci-flat
	manifolds are \emph{locally hyperk\"{a}hler} and they are
	of fundamental importance in the theory of 
	gravitational instantons; furthermore, if we assume that $M$ is compact,
	$(M,g)$ has Ricci-parallel twistor space if and only if it is,
	up to quotients, $\SS^4$, $\cp^2$, a flat torus or a K3 surface
	(\cite{frikur, hitch}).  
	
	\section{Preliminaries and notation}
	In this section, we recall the basic properties
	of four-dimensional Riemannian manifolds and their
	associated twistor spaces, exploiting
	the moving frame methods: for a more
	detailed dissertation of these subjects, we might
	refer the reader to \cite{alimasrig, besse, cmbook, 
	damenothesis, jenrig, salamon}.
	Throughout the paper, we adopt Einstein's summation
	convention over repeated indices. 

	Let $(M,g)$ be 
	a Riemannian manifold of dimension $n$, which
	we will always assume connected and oriented: if 
	$(U,x^1,...,x^n)$ is a local open chart of $M$, 
	we can always define a \emph{local orthonormal coframe} on $U$,
	i.e. a set $\{\omega^i\}_{i=1}^n$ of $1$-forms which
	are orthonormal with respect to the inner product induced
	by $g$ on $\Lambda^1$, in order to have that $g$ can be
	expressed as $g=\delta_{ij}\omega^i\otimes\omega^j$. 
	The Levi-Civita connection can be locally represented 
	by a set $\set{\omega_j^i}_{i,j=1}^n$ of $1$-forms
	satisfying the \emph{Cartan's first structure equation}
	\begin{equation} \label{cartanfirst}
		d\omega^i=-\omega_j^i\wedge\omega^j
	\end{equation} 
	and $\omega_j^i+\omega_i^j=0$: these are the so-called
	\emph{Levi-Civita connection forms} and
	can be seen as an $\mathfrak{so}(n)$-valued 1-form, where
	$\mathfrak{so}(n)$ is the Lie algebra of $O(n)$ or,
	equivalently, of $SO(n)$
	(\cite{alimasrig, spivak}). The exterior derivative
	of the Levi-Civita connection forms allows us
	to define the \emph{curvature forms}, which
	are $2$-forms $\Omega_j^i$ that encompass all the local
	information about the curvature: indeed, they appear in
	the \emph{Cartan's second structure equation}
	\begin{equation} \label{cartansecond}
		d\omega_j^i=-\omega_k^i\wedge\omega_j^k+\Omega_j^i
	\end{equation}
	and it can be shown that (\cite{cmbook})
	\begin{equation} \label{riemcompforms}
		\Omega_j^i=\dfrac{1}{2}R_{jkl}^i\omega^k\wedge\omega^l,
	\end{equation}
	where $R_{jkl}^i=R_{ijkl}$ are the local components of the
	Riemann curvature tensor with respect to $\{\omega^i\}$
	(the equality on the components holds since the
	chosen coframe is orthonormal). It is
	well known that, if $n\geq 3$, 
	the Riemann curvature tensor can be
	decomposed into a sum of orthogonal pieces (in fact,
	the space of algebraic curvature tensors splits
	into a sum of orthogonal subspaces under the action of
	$O(n)$ \cite{besse}):
	\begin{equation}\label{riemdeco}
		\riem=\weyl+\dfrac{1}{n-2}\operatorname{E}
		\KN g-\dfrac{S}{2n(n-1)}g\KN g,
	\end{equation}
	where $\weyl$ is the \emph{Weyl tensor}, $\operatorname{E}:=
	\ricc-\frac{S}{n}g$ is the 
	\emph{traceless Ricci tensor}, $S$ is the \emph{scalar curvature}
	and $\KN$ is the so-called \emph{Kulkarni-Nomizu product}.
	
	Now, let $n=4$: the action of
	the Hodge $\star$ operator on the bundle $\Lambda^2$ of 
	2-forms induces the well-known splitting
	\begin{equation} \label{lambdasplit}
	\Lambda^2=\Lambda_+\oplus\Lambda_-,
	\end{equation}
	where $\Lambda_+$ (resp., $\Lambda_-$) is the
	bundle of \emph{self-dual} (resp. \emph{anti-self-dual})
	forms, meaning that $\eta\in\Lambda_{\pm}$ if 
	$\star\eta=\pm\eta$. Given a local orthonormal
	coframe $\{\omega^i\}_{i=1}^4$ defined on a local
	chart $U\subset M$, we can locally define the
	\emph{Riemann curvature operator} as
	\begin{align} \label{riemcurvop}
		\mathcal{R}&:\Lambda^2\longra\Lambda^2 \\
		\mathcal{R}(\eta)&=\dfrac{1}{4}R_{ijkl}\eta_{kl}\omega^i\wedge\omega^j, \notag
	\end{align}
	where $\eta=\frac{1}{2}\eta_{ij}\omega^i\wedge\omega^j$: similarly,
	we can define the linear operators
	\begin{align}
		\mathcal{W}(\eta)&=
		\dfrac{1}{4}W_{ijkl}\eta_{kl}\omega^i\wedge\omega^j, \label{weylop}\\
		\mathcal{T}(\eta)&=\dfrac{1}{4}\pa{\operatorname{E}\KN g}_{ijkl}
		\eta_{kl}\omega^i\wedge\omega^j, \label{riccop}
	\end{align}
	where $W_{ijkl}$ are the components of $\weyl$. 
	A simple observation shows
	that
	\[
	\mathcal{W}(\Lambda_{\pm})\subset\Lambda_{\pm} \quad\mbox{and}\quad
	\mathcal{T}(\Lambda_{\pm})\subset\Lambda_{\mp}:
	\]
	since $\mathcal{W}$ preserves the subbundles of $\Lambda^2$, 
	this induces the splitting
	\[
	\mathcal{W}=\mathcal{W}^++\mathcal{W}^-, \quad\mathcal{W}^{\pm}=
	\dfrac{1}{2}\sq{\mathcal{W}\pm\star\mathcal{W}}
	\in\operatorname{End}(\Lambda_{\pm}),
	\]
	which provides the well-known decomposition
	\begin{equation} \label{weyldeco}
		\weyl=\weyl^++\weyl^-.
	\end{equation}
	We say that $\weyl^+$ (resp. $\weyl^-$) is the
	\emph{self-dual} (resp. \emph{anti-self-dual}) part of
	the Weyl tensor: $(M,g)$ is said to be 
	\emph{self-dual} (resp. \emph{anti-self-dual}) if
	$\weyl^-\equiv 0$ (resp. $\weyl^+\equiv 0$) on $M$ and
	\emph{half conformally flat} if one part of $\weyl$
	vanishes identically on $M$. 
	
	The splitting of $\Lambda^2$ henceforth 
	induces a decomposition of the Riemann curvature operator
	\eqref{riemcurvop}, which, by, \eqref{riemdeco},
	\eqref{weylop} and \eqref{riccop},
	can be 
	visualized as the symmetric matrix
	\begin{equation} \label{riemoperator}
	\mathcal{R}=
	\left(
	\begin{array}{cc}
		\mathcal{W}^++\frac{S}{12}\operatorname{Id}_{\Lambda_+} & \frac{1}{2}
		\mathcal{T}\restrict{\Lambda_-} \\
		\frac{1}{2}
		\mathcal{T}\restrict{\Lambda_+} & 
		\mathcal{W}^-+\frac{S}{12}\operatorname{Id}_{\Lambda_-}
	\end{array}
	\right).
	\end{equation}
	Locally,
	with respect to the chosen local orthonormal coframe,
	we can express $\mathcal{R}$ as
	\begin{equation} \label{riemmatrix}
	\mathcal{R}=
	\left(
	\begin{array}{cc}
		A & B^T \\
		B & C
	\end{array}
	\right), 
	\end{equation}
	where $A$, $B$ and $C$ are obtained by evaluating $\mathcal{R}$
	at the standard bases of $\Lambda_{\pm}$ and 
	depend on the components of $\riem$ (the
	explicit expressions can be found, e.g., in \cite{cdmnachr, damenothesis, 
	jenrig}). By \eqref{riemoperator} and \eqref{riemmatrix}, 
	we obtain a nice characterization of
	half conformally flat and Einstein metrics: 
	\begin{proposition} \label{charactmetric}
		Let $(M,g)$ be a Riemannian manifold of dimension $n=4$. Then, 
		\begin{itemize}
			\item $(M,g)$ is self-dual if and only 
			if $C=(S/12)I_3$ (resp. $A=(S/12)I_3)$
			at every point $p\in M$ for every positively (resp. negatively)
			oriented local orthonormal frame $\{e_i\}$ around $p$;
			\item $(M,g)$ is 
			anti-self-dual if and only if 
			$A=(S/12)I_3$ (resp. $C=(S/12)I_3)$
			at every point $p\in M$ for every positively (resp. negatively)
			oriented local orthonormal frame $\{e_i\}$ around $p$;
			\item $(M,g)$ is an Einstein manifold
			if and only if $B=0$ at every $p\in M$ 
			for every local orthonormal frame
			around $p$.
		\end{itemize}
	\end{proposition}

	Now, we introduce the notion of \emph{Riemannian twistor space} associated
	to a four-manifold: this is defined as 
	the associated bundle (see \cite{kobnom} for
	a complete dissertation about this construction)
	\begin{equation}
		Z:=O(M)_-\times_{SO(4)}\left(
		\bigslant{SO(4)}{U(2)}\right)=\bigslant{O(M)_-}{U(2)},
	\end{equation}
	where $O(M)_-$ is the principal $SO(4)$-bundle of negatively
	oriented local orthonormal frames over $(M,g)$ and $U(2)$ is
	a Lie subgroup of $SO(4)$ isomorphic to the unitary group. 
	This construction can
	be generalized starting from every Riemannian manifold of
	even dimension (\cite{debnan, jenrig}): furthermore, 
	the choice of negatively oriented frames is just a convention, since
	we can equivalently define the twistor space $Z$ starting from
	the positively oriented frame bundle $O(M)_+$. We will denote
 	the bundle projection $Z\longra M$ as $\pi$ and the generic fiber 
 	$SO(4)/U(2)\cong\cp^1$ of the twistor bundle as 
 	$F$:
	by definition, $Z$ is a smooth manifold of dimension $6$ and 
	it can be shown that the fibers of the twistor bundle 
	parametrize the $g$-orthogonal complex structures on the
	tangent spaces of $M$ (\cite{athisin}).
	
	The Levi-Civita connection induces a splitting of the tangent
	space of $O(M)_-$ into a horizontal subspace $H_e$ and a vertical 
	subspace $V_e$ 
	at every point $e\in O(M)_-$, 
	which is preserved on $Z$: this 
	means that, if $\sigma:O(M)_-\longra Z$ is the bundle
	projection of the principal $U(2)$-bundle over $Z$ given by $O(M)_-$,
	then, at every $q\in Z$, 
	\[
	T_qZ=\mathscr{H}_q\oplus\mathscr{V}_q:=
	\sigma_{\ast e}(H_e)\oplus\sigma_{\ast e}(V_e),
	\]
	where $e\in\sigma^{-1}(q)$. The collections of 
	the subspaces $\mathscr{H}_q$ and 
	$\mathscr{V}_q$ define smooth distributions of $TZ$, called, respectively,
	horizontal and vertical distributions. 
	
	It is well-known that there exists a natural $1$-parameter family
	of Riemannian metrics on $Z$ (see, e.g., 
	\cite{damenothesis, davmussurv, debnan, jenrig}) as
	\begin{equation} \label{metrictwist}
		g_t=\pi^{\ast}(g)+t^2\metric,
	\end{equation}
	where $\metric$ denotes the standard $SO(4)$-invariant
	metric on the homogeneous space $F$ (\cite{debnan}).
	Given
	a local orthonormal coframe $\{\omega^i\}$ on $M$ 
	(seen as a pullback of the tautological $1$-form on $O(M)$
	\emph{via} a smooth section) and the corresponding
	Levi-Civita connection forms $\{\omega_j^i\}$, we can locally write
	the metrics $g_t$ in \eqref{metrictwist} as (\cite{cdmnachr,jenrig})
	\begin{equation} \label{metriclocal}
		g_t=\sum_{a=1}^4\pa{\theta^a}^2+4t^2\sq{\pa{\theta^5}^2+
		\pa{\theta^6}^2},
	\end{equation}
	where 
	\begin{equation} \label{twistcoframe}
		\theta^a:=\omega^a, \mbox{ } a=1,...,4, \quad
		\theta^5:=t\pa{\omega_3^1+\omega_2^4}, \quad
		\theta^6:=t\pa{\omega_4^1+\omega_3^2}
	\end{equation}
	is a local orthonormal coframe on $Z$ such that, if 
	$\{\bar{e}_p\}_{p=1}^6$ is the dual local frame, $\bar{e}_a$, $a=1,...,4$,
	are horizontal vector fields and $\bar{e}_5,\bar{e}_6$ are vertical 
	vector fields (for the sake of simplicity, we omit the pullback
	notation in \eqref{metriclocal} and \eqref{twistcoframe},
	as in \cite{jenrig}). 
	
	By means of \eqref{twistcoframe} and \eqref{cartanfirst}, 
	we can compute the 
	Levi-Civita connection forms as
	\begin{align} \label{twistconnform}
		\theta_b^a&=\omega_b^a+\dfrac{t}{2}
		\pa{\qt{ba}\theta^5+\qq{ba}\theta^6}, \\
		\theta_a^5&=\dfrac{t}{2}\qt{ab}\theta^b, \notag\\
		\theta_a^6&=\dfrac{t}{2}\qq{ab}\theta^b, \quad
		a=1,...,4 \notag,
	\end{align}
	where $\qt{ab}:=R_{13ab}+R_{42ab}$ and $\qq{ab}:=R_{14ab}+R_{23ab}$;
	using \eqref{twistconnform} and \eqref{cartansecond}, one
	can compute all the local components of the curvature tensors
	of $(Z,g_t)$ (see the Appendices of \cite{cdmagag} for a full
	description). 
	
	An important fact about the metrics $g_t$ is that the
	twistor projection $\pi:(Z,g_t)\longra (M,g)$ becomes
	a \emph{Riemannian submersion}, i.e. the horizontal
	tangent spaces $\mathscr{H}_q$ are isometric to the
	tangent spaces $T_{\pi(q)}M$ \emph{via} $\pi_{\ast}$:
	moreover, this submersion has \emph{totally geodesic fibers}, 
	since a simple computation using \eqref{twistconnform} shows
	that the tensor $\operatorname{T}$ introduced by O'Neill
	(\cite{oneill}), which
	is precisely the second fundamental form of each fiber, 
	identically vanishes on $Z$. 
	
	\section{Proof of Theorem \ref{constscaltwist}}
		First, recall that, by construction, the fibers of
		the twistor bundle parametrize local orthonormal (co)frames
		on $M$, up to a $U(2)$-action. It is well known that
		matrices in $SO(4)$ represent smooth, orientation-preserving
		changes of (co)frames on a four-dimensional oriented manifold:
		therefore, there exists a natural action of $SO(4)$ on 
		$\Lambda^2$, locally given by 
		\begin{align*}
			SO(4)\times\Lambda^2&\longra\Lambda^2\\
			(a,\omega^i\wedge\omega^j)&\longmapsto
			a(\omega^i)\wedge a(\omega^j),
		\end{align*}
		where $a(\omega^i):=(a^{-1})_j^i\omega^j$ is the natural
		action of $SO(4)$ on $\Lambda^1$. We want to restrict 
		this action on the subbundles $\Lambda_{\pm}$: first,
		we note that $\mathfrak{so}(4)$ is the
		only semisimple Lie algebra of an orthogonal group, since
		there exist two proper ideals of $\mathfrak{so}(4)$, both isomorphic
		to $\mathfrak{so}(3)$, i.e.
		\[
		\mathfrak{so}(4)\cong\mathfrak{so}(3)\oplus\mathfrak{so}(3)
		\]
		(recall that $\mathfrak{so}(n)$ is a simple Lie algebra if
		$n\neq 4$). It is easy to show that $\mathfrak{so}(4)\cong
		\Lambda^2$ and that, consequently, 
		$\mathfrak{so}(3)\cong\Lambda_{\pm}$: by standard theory of 
		coverings, there exists a $2:1$ surjective Lie group homomorphism
		\begin{align} \label{homom}
			\mu:SO(4)&\longra SO(3)\times SO(3)\\
			a&\longmapsto (a_+,a_-), \notag
		\end{align}
		where $a_{\pm}$ is the restriction of the action of $a$ on
		$\Lambda^2$ to $\Lambda_{\pm}$ (for a full dissertation,
		see \cite{besse, salamon}). The curvature forms $\Omega_j^i
		\in\Lambda^2$
		defined by \eqref{cartansecond} transform as
		\[
		\tilde{\Omega}_j^i=(a^{-1})_k^i\Omega_l^ka_j^l
		\]
		under the smooth change of frames defined by $a\in SO(4)$:
		the restriction \eqref{homom} induces analogous transformation
		laws for the blocks of \eqref{riemmatrix}, i.e.
		\begin{equation} \label{transflaw}
			\tilde{A}=a_+^{-1}Aa_+, \quad
			\tilde{B}=a_-^{-1}Ba_+, \quad \tilde{C}=a_-^{-1}Ca_-,
		\end{equation}
		where $(a_+,a_-)=\mu(a)$.
		
		Now, recall the expression of the scalar curvature $\ol{S}$
		of $(Z,g_t)$ in terms of the Riemann 
		curvature tensor of $M$ (\cite{cdmagag, jenrig}):
		\begin{equation} \label{scaltwist}
			\ol{S}=S+\dfrac{2}{t^2}-\dfrac{t^2}{4}(\abs{\qt{ab}}^2+\abs{\qq{ab}}^2),
		\end{equation}
		where $\abs{\qt{ab}}^2=\sum_{a,b=1}^4\qt{ab}\qt{ab}$ and
		$\abs{\qq{ab}}^2=\sum_{a,b=1}^4\qq{ab}\qq{ab}$.
		Since $S$ is constant by hypothesis and 
		$\ol{S}$ is constant on the fibers of the twistor bundle, 
		we obtain that 
		$\abs{\qt{ab}}^2+\abs{\qq{ab}}^2$ does not
		depend on the chosen local orthonormal frame on $M$; 
		hence, at every point
		$p\in M$, 
		we can choose a local orthonormal frame such that $A$ is in diagonal form.
		
		Now, since $M$ is Einstein, a simple computation shows that
		\[
		\abs{\qt{ab}}^2=4(A_{22})^2 \quad \mbox{ and } \quad
		\abs{\qq{ab}}^2=4(A_{33})^2;
		\] 
		therefore, \eqref{scaltwist} becomes
		\[
		A_{22}^2+A_{33}^2=\dfrac{1}{t^2}\pa{S+\dfrac{2}{t^2}-\ol{S}}.
		\]
		Since the right-hand side does not depend on the chosen frame, we can 
		rotate the frame with a suitable choice of 
		$a_+$ and use the
		transformation laws \eqref{transflaw} 
		in order to obtain the equation
		\[
		\tilde{A}_{22}^2+
		\tilde{A}_{33}^2=
		A_{11}^2+A_{33}^2=\dfrac{1}{t^2}\pa{S+\dfrac{2}{t^2}-\ol{S}},
		\]
		which implies $A_{11}^2=A_{22}^2$: for instance,
		it suffices to choose a smooth change of frames, determined
		by a matrix $a\in SO(4)$, such that the corresponding
		restriction $a_+\in SO(3)$ defined by
		the homomorphism $\mu$ in \eqref{homom}
		is
		\[
		a_+=\pmat{
		0 & -1 & 0\\
		1 & 0 & 0\\
		0 & 0 & 1
		}.
		\]
		A similar line of reasoning allows us to
		conclude that
		\[
		A_{11}^2=A_{22}^2=A_{33}^2.
		\] 
		We know that $\operatorname{tr}A=S/4$, which means that, up to
		frame rotations, we have just two 
		possibilities:
		\begin{equation} \label{matrposs}
		A=\pmat{
			\dfrac{S}{12} & 0 & 0\\
			0 & \dfrac{S}{12} & 0\\
			0 & 0 & \dfrac{S}{12}}
		\quad \mbox{ or } \quad
		A=\pmat{
			-\dfrac{S}{4} & 0 & 0\\
			0 & \dfrac{S}{4} & 0\\
			0 & 0 & \dfrac{S}{4}
		}.
		\end{equation}
		Now, let us assume that there exists $p\in M$ such that 
		the operator $A$ has eigenvalues $-S/4$, $S/4$, $S/4$, i.e. the 
		matrix $A$ can be diagonalized as in the second case of 
		\eqref{matrposs}: this means that the set
		\[
		\mathcal{A}:=\{p\in M\mbox{ : } \abs{\weyl}^-\neq 0\}
		\]
		is non-empty. Given a linear, symmetric operator $T:
		\Lambda^2\longra\Lambda^2$, we denote 
		the set of distinct eigenvalues of $T$ as 
		$\operatorname{spec}(T)$: note that, since $A$ can only be 
		diagonalized in two ways, according to \eqref{matrposs}, 
		we have $\#\operatorname{spec}(\weyl^-)=\#\operatorname{spec}(A)
		\leq 2$ on $M$. 
		By hypothesis, $(M,g)$ is an Einstein manifold and,
		therefore, $\diver\pa{\weyl^-}=0$: a celebrated result
		due to Derdzi\'{n}ski states that the Riemannian metric
		\[
		\tilde{g}=\pa{24\abs{\weyl_g^-}_g^2}^{\frac{1}{3}}g,
		\]
		which is conformal to $g$, 
		is a K\"{a}hler metric on $\mathcal{A}$, with
		respect to the opposite orientation of $M$ (\cite{derd}),
		and that the K\"{a}hler form of $\tilde{g}$ corresponds
		to the single eigenvalue of $\weyl^-$. Moreover,
		since $(M,g)$ is Einstein, the same result assures that 
		$\mathcal{A}=M$: since, in our case, $\abs{\weyl^-}^2$ is 
		a constant function, we conclude that $g$ is a K\"{a}hler-Einstein
		metric on $M$. However, this is impossible: indeed, 
		since $g$ is not a self-dual metric by assumption, the eigenvalues
		of $\weyl^-$ for such a metric must be 
		$S/6$, $-S/12$, $-S/12$ (\cite{derd, gray}), which would imply
		that the eigenvalues of $A$ are $S/4$, $0$, $0$: this
		could only be possible if $S=0$, but this is in contradiction
		with the fact that $(M,g)$ is not self-dual. Hence, the
		second case in \eqref{matrposs} cannot occur, which means
		that $(M,g)$ is a self-dual manifold and this concludes
		the proof. 
	\begin{rem}
		The last part of the proof can be rewritten in the following way:
		by \cite[Proposition 5]{derd}, we know that, for every K\"{a}hler
		metric on a four-manifold, $\operatorname{sign}(S)=
		\operatorname{sign}\pa{\operatorname{det}\weyl^-}$, 
		showing immediately
		that the second case in \eqref{matrposs} occurs if and only
		if $S=0$. 
	\end{rem}
	\begin{rem}
		Let $\ol{S}$ be constant on $Z$: a straightforward computation
		shows that, on $M$, with respect to a negatively oriented frame,
		\[
		\abs{\weyl^-}^2=\abs{\qd{ab}}^2+\abs{\qt{ab}}^2+
		\abs{\qq{ab}}^2-\dfrac{S^2}{12},
		\] 
		where $\qd{ab}=R_{12ab}+R_{34ab}$ and 
		$\abs{\qd{ab}}^2=\sum_{a,b=1}^4\qd{ab}\qd{ab}$. 
		By \eqref{scaltwist} and using \eqref{transflaw}
		with suitable choices of $a_+$, it is easy to show that
		$\abs{\weyl}^2$ is a constant function on $M$. 
		Recall that Einstein four-manifolds satisfy the
		following Bochner-Weitzenb\"{o}ck formula (\cite{cmbook, derd}):
		\begin{equation} \label{bochweitz}
			\dfrac{1}{2}\Delta\abs{\weyl^-}^2=
			\abs{\nabla\weyl^-}^2+\dfrac{S}{2}\abs{\weyl^-}^2-
			18\det\weyl^-
		\end{equation}
		(in fact, this is true more generally for \emph{half
		harmonic Weyl} metrics, i.e. the ones such that
	$\diver\weyl^-\equiv 0$ on $M$).
		If $S$ is positive, then, since $\abs{\weyl^-}^2$ is constant,
		by \eqref{bochweitz} we immediately obtain that
		$\det\weyl^-\geq 0$: hence, we can exploit a well-known result
		due to Wu (\cite{wudet}) to conclude that the universal
		cover of $(M,g)$
		is either self-dual or conformally K\"{a}hler with respect
		to a complex structure compatible with the opposite orientation.
		However, the second case is impossible due to the hypothesis
		on $\ol{S}$, as deduced in the proof of Theorem \ref{constscaltwist},
		hence we obtain again that $(M,g)$ is self-dual.
	\end{rem}
	Now, it is possible to easily prove Corollary \ref{corollaryeinst}:
	\begin{proof}[Proof of Corollary \ref{corollaryeinst}]
		Up to rescaling $g$ by a positive constant, we may assume that $S=12/t^2$, which,
		by Friedrich and Kurke's result \cite{frikur}, implies
		that $(Z,g_t,J)$ is a K\"{a}hler-Einstein manifold with positive 
		scalar curvature $\ol{S}=12/t^2$ (it is apparent from
		\eqref{scaltwist}). The compact case follows immediately from
		\cite{frikur, hitch2}.
	\end{proof}
	\begin{rem}
		A long-standing conjecture asserts that 
		the only compact 
		Einstein four-manifolds with positive \emph{sectional curvature}
		are $\SS^4$ and $\cp^2$, both endowed with their standard metrics: in
		other words, by \cite{frikur, hitch2}, 
		a compact Einstein four-manifold with positive sectional 
		curvature
		has to be half conformally flat. 
		In this direction, it is worth mentioning the work of Gursky and LeBrun,
		who characterized $\cp^2$ under a cohomological assumption
		(\cite{gurskylebrun}), and Yang (\cite{yang}), who proved the conjecture 
		assuming a pinching condition on the sectional curvature 
		which was later improved (see e.g. \cite{caotran, costa, 
			ribeiro}). We observe
		that Theorem \ref{constscaltwist} and Corollary 
		\ref{corollaryeinst} allow us to rephrase
the conjecture about Einstein four-manifolds with positive
sectional curvature as follows: given a compact Einstein manifold $(M,g)$
of dimension four with positive sectional curvature, then
its twistor space $(Z,g_t)$ has constant scalar curvature. We recall that,
to the best of our knowledge, the only result in the literature concerning
twistor spaces of Einstein four-manifolds with nowhere vanishing sectional
curvature is due to Reznikov (\cite{reznikov}).
	\end{rem}
	\section{Ricci parallel twistor spaces}
	In this section, we classify complete four-manifolds 
	with Ricci parallel twistor space.
	We recall a well-known result, which is a natural consequence
	of the De Rham decomposition Theorem (see, for instance, 
	\cite{boubelbergery, wu}):
	\begin{theorem} \label{splittingricpar}
		Let $(M,g)$ be a complete, simply connected, 
		Ricci parallel manifold of dimension $n$.
		Then, $M$ splits canonically into a Riemannian product of 
		Einstein manifolds. More precisely, there exist unique 
		$\lambda_1,...,\lambda_N\in\RR$, $(M_1,g_1),...,(M_N,g_N)$ 
		Einstein manifolds, such that $\ricc_i=\lambda_ig_i$ for
		every $i=1,...,N$, and an isometry 
		\[
		f:M\longrightarrow\Pi_{i=1}^N(M_i,g_i),
		\]
		which is unique up to composition with a product of isometries
		of the factors $M_i$.
	\end{theorem}
	\begin{rem}
		\begin{enumerate}
			\item If $M$ is not simply connected, Theorem \ref{splittingricpar}
			guarantees that $M$ splits locally as a product of Einstein
			manifolds;
			\item there exists a more general result, due to Eisenhart
			\cite{eisenhart}: if $(M,g)$ is a complete, simply connected
			Riemannian manifold admitting
			a symmetric, $(0,2)$-tensor field $\operatorname{T}$ 
			which is parallel with respect
			to the Levi-Civita connection, then either $\operatorname{T}$ is
			proportional to $g$ or $(M,g)$ splits isometrically as a product
			manifold. 
		\end{enumerate}
	\end{rem}
	\noindent
	By this splitting result, we can finally prove Theorem \ref{twistricpar},
	which generalizes the classification result for Einstein twistor
	spaces due to Friedrich and Grunewald (\cite{frigru}).
	\begin{proof}[Proof of Theorem \ref{twistricpar}]
		Let us assume that $(Z,g_t)$ is Ricci parallel but not Einstein.
		Recall that $Z$ is constructed as an associated bundle
		starting from the principal $SO(4)$-bundle $O(M)_-$ and
		that the metrics $g_t$ are such that the twistor projection
		is, in fact, a Riemannian submersion with totally geodesic fibers:
		a result due to Vilms tells us that $g_t$ is complete 
		for every $t>0$ (\cite{vilms}). Moreover, by definition
		of a fiber bundle we know that, for every $p\in M$, 
		there exists an open neighborhood $U_p$ of $p$ such that
		\begin{equation} \label{bundlediffeo}
			\pi^{-1}(U_p)\cong U_p\times F.
		\end{equation}
		Therefore, up to choosing $U'_p\subset U_p$, by Theorem
		\ref{splittingricpar} we know that 
		\eqref{bundlediffeo} is an isometry and it is unique, up to
		composition with a product of isometries of $U_p$ and 
		the fiber of $\pi$: moreover, we can also assume
		that $U_p$ is a local chart of $M$. 
		
		Let $\{e_q'\}_{q=1}^6$ be an orthonormal frame on 
		$U_p\times F$ such that $e_a'\in TU_p$ for $a=1,...,4$ and
		$e_5'$, $e_6'\in TF$: we choose a local orthonormal 
		$\{e_i\}_{i=1}^4$ on $U_p$ such that, if 
		$f:\pi^{-1}(U_p)\rightarrow U_p\times F$ 
		is the isometry given by Theorem \ref{splittingricpar},
		$f_{\ast}(\ol{e}_q)=e_q'$ for every $q$, where
		$\{\ol{e}_q\}_{q=1}^6$ is the local orthonormal frame
		on $\pi^{-1}(U_p)$ constructed as the dual of \eqref{twistcoframe},
		starting from $\{e_i\}_{i=1}^4$. Since $f$ is an isometry and
		the metric on $U_p\times F$ is the product metric, 
		if $(\theta')_r^q$ are the connection forms
		associated to $\{e_q'\}_{q=1}^6$, we have that
		\begin{align} \label{connformvanish}
			0&=(\theta')_5^a(e_p')e_a'=\theta_5^b(e_p)e_b\\
			0&=(\theta')_6^a(e_p')e_a'=\theta_6^b(e_p)e_b\notag
		\end{align}
		for every $p$ and for every $a,b=1,...,4$, 
		where $\theta_q^r$ are the Levi-Civita
		connection forms associated to $\{\ol{e}_q\}_{q=1}^6$: 
		by \eqref{twistconnform}, we immediately obtain that
		\[
		\qt{ab}=\qq{ab}=0.
		\]
		Now, we note that, if we rotate the frame $\{e_i\}_{i=1}^4$ by
		means of a smooth change of frames $a\in SO(4)$, 
		the vanishing condition \eqref{connformvanish} still holds
		for the new frame: 
		moreover, it can be easily shown that
		the equations $\qt{ab}=\qq{ab}=0$ can be
		rephrased as 
		\[
		A_{12}=A_{13}=A_{23}=A_{22}=A_{33}=B_{12}=B_{13}=B_{22}=B_{23}=B_{32}=B_{33}=0.
		\]
		Hence, if, for instance, we choose
		$a$ in such a way that
		\[
		a_+=\pmat{
			0 & -1 & 0\\
			1 & 0 & 0\\
			0 & 0 & 1
		}
		\quad\mbox{ and }\quad
		a_-=I_3,
		\]
		by \eqref{transflaw} we also obtain
		that the components of the new matrices $\tilde{A}$ and
		$\tilde{B}$ satisfy
		\[
		0=\tilde{A}_{22}=A_{11}, \quad 0=\tilde{B}_{12}=-B_{11}, \quad
		0=\tilde{B}_{22}=-B_{21}, \quad 0=\tilde{B}_{32}=-B_{31},
		\] 
		which implies that $A$ and $B$ are null matrices.
		Since this holds for some orthonormal frame on $U_p$, by
		\eqref{transflaw} it has
		to be true for every frame; repeating this argument for
		every point on $M$, we get that 
		$A=B=0$ on $M$, which, by Proposition \ref{charactmetric}, 
		means that $(M,g)$ is
		self-dual and Ricci-flat.
	\end{proof}
	As a consequence, we can classify complete, locally symmetric
	twistor spaces: indeed, by the argument used in \cite{cdmagag} and
	recalling that every Ricci-flat, homogeneous manifold is flat
	(\cite{alek}), 
	we can state the following
	\begin{cor}
		Let $(M,g)$ be a complete Riemannian four-manifold and 
		$(Z,g_t)$ be its twistor space. Then $(Z,g_t)$ is locally
		symmetric if and only if $(M,g)$ is the standard four-sphere
		$\SS^4$ or a flat manifold. 
	\end{cor}
	\begin{rem}
		Recall that, by Friedrich-Grunewald and Hitchin's 
		classification result
		(\cite{frigru, hitch2}), if $M$ is also
			compact we conclude that $(Z,g_t)$ is Ricci parallel if and 
			only if $M$ is $\SS^4$, $\cp^2$, a flat torus or a K3 surface,
			up to quotients: in the last two cases, i.e. when 
			$Z$ is Ricci parallel but not Einstein, we also obtain
			that the twistor bundle is trivial (this holds
			in general for the twistor spaces of Ricci-flat, self-dual 
			four-manifolds, as observed in \cite{hitchinpoly}).
	\end{rem}
	As we recalled, the metrics $g_t$ in \eqref{metrictwist} 
	are constructed in such a way
	that the twistor projection becomes a Riemannian submersion
	with totally geodesic fibers. We might exploit
	the properties of Riemannian submersions in order to 
	provide an alternative proof of Theorem \ref{twistricpar}. 
	Let $\mathscr{H}$ be the \emph{horizontal distribution}
	of a submersion $\pi$: we know that $\mathscr{H}$
	is integrable if and only if the $(2,1)$-tensor field 
	$\operatorname{A}$, defined as (see \cite{oneill})
	\begin{equation} \label{tensorasub}
		A(X,Y)=\pi_\mathscr{H}\pa{\nabla_{\pi_\mathscr{H}X}(\pi_\mathscr{V}Y)}
		-\pi_\mathscr{V}\pa{\nabla_{\pi_\mathscr{H}X}(\pi_\mathscr{H}Y)}
	\end{equation}
	vanishes for every smooth vector field $X$ and $Y$, where
	$\pi_\mathscr{H}$ (respectively, $\pi_\mathscr{V}$) denotes
	the projection of vector fields onto the horizontal (respectively,
	onto the vertical) subbundle of $TM$. If we consider
	a Riemannian four-manifold and its twistor space, we
	easily obtain the following
	\begin{theorem} \label{horizintegr}
		Let $(M,g)$ be a Riemannian four-manifold and $(Z,g_t)$ its
		twistor space. Then the horizontal distribution $\mathscr{H}$
		of the submersion $\pi$ is integrable if and only if
		$(M,g)$ is self-dual and Ricci-flat. 
	\end{theorem}
	\begin{proof}
		By \eqref{tensorasub}, we can find the local expression of the
		components of $\operatorname{A}$
		\begin{align*}
			A_{ab}^pe_p&=A(e_a,e_b)=-\pi_\mathscr{V}\pa{\nabla_{e_a}e_b}=
			-\pi_\mathscr{V}\pa{\theta_b^p(e_a)e_p}=-\theta_b^5(e_a)e_5
			-\theta_b^6(e_a)e_6\\
			A_{a5}^pe_p&=A(e_a,e_5)=\pi_\mathscr{H}\pa{\nabla_{e_a}e_5}=
			-\pi_\mathscr{H}\pa{\theta_5^p(e_a)e_p}=-\theta_5^b(e_a)e_b
			\notag\\
			A_{a6}^pe_p&=A(e_a,e_6)=\pi_\mathscr{H}\pa{\nabla_{e_a}e_6}=
			-\pi_\mathscr{H}\pa{\theta_6^p(e_a)e_p}=-\theta_6^b(e_a)e_b
			\notag\\
			A(e_5,e_a)&=A(e_6,e_a)=A(e_5,e_6)=A(e_6,e_5)=0, \notag
		\end{align*}
		which, by \eqref{twistconnform}, imply 
		\begin{align} \label{compasub}
			A_{ab}^5&=\dfrac{t}{2}\qt{ab}, \quad 
			A_{ab}^6=\dfrac{t}{2}\qt{ab}\\
			A_{a5}^b&=-\dfrac{t}{2}\qt{ab}, \quad
			A_{a6}^b=-\dfrac{t}{2}\qq{ab} \notag\\
			A_{ab}^c&=A_{5a}^p=A_{6a}^p=A_{56}^p=A_{65}^p=0, \notag
		\end{align}
		where $a,b,c=1,...,4$ and $p=1,...,6$. 
		It is clear that $\operatorname{A}\equiv0$ on $Z$ if and only if
		$\qt{ab}=\qq{ab}=0$ for every local orthonormal frame on $M$,
		which, by Proposition \ref{charactmetric} and
		\eqref{transflaw}, 
		holds if and only if $(M,g)$ is self-dual and Ricci-flat. 
	\end{proof}
	By means of Theorem \ref{horizintegr}, it is possible
	to provide an alternative proof for the classification of
	Ricci parallel twistor spaces.
	\begin{proof}[Second proof of Theorem \ref{twistricpar}]
		As before, let us assume that $(Z,g_t)$ is Ricci parallel,
		but not Einstein: moreover, up to passing to the universal
		covering, we may assume that $M$ is simply connected, which
		also guarantees that $Z$ is simply connected. By 
		Theorem \ref{splittingricpar}, $(Z,g_t)$ is isometric
		to a product manifold and, since $\pi$ is a Riemannian submersion,
		for every $q\in Z$, the horizontal subspace $\mathscr{H}_q$ 
		is isometric to $T_{\pi(q)}M$ \emph{via} $\pi_{\ast}$, hence implying
		that $Z$ is isometric to $M\times F$ and that $(M,g)$ is
		an Einstein manifold. Exploiting this
		argument, we can conclude that the isometry
		between $\mathscr{H}_q$ and $T_{\pi(q)}M$ and the
		fact that the twistor space globally splits
		isometrically as the product $M\times F$, where
		the metric on the factor $M$ is exactly $g$, 
		guarantee that the horizontal distribution is integrable. 
		By Theorem \ref{horizintegr}, 
		the claim is proven. 
	\end{proof}
	We conclude by summarizing the results
	concerning Ricci parallel twistor spaces as follows:
	\begin{theorem}
		Let $(M,g)$ be a complete Riemannian four-manifold and
		$(Z,g_t)$ be its twistor space. Then, if
		$(Z,g_t)$ is not an Einstein manifold, the following
		conditions are equivalent:
		\begin{enumerate}
			\item $(Z,g_t)$ is a Ricci parallel manifold;
			\item the horizontal distribution $\mathscr{H}$ of
			the Riemannian submersion $\pi$ is integrable;
			\item $(M,g)$ is a Ricci-flat, self-dual (and then,
			locally hyperk\"{a}hler) manifold. 
		\end{enumerate}
	In particular, if $M$ is compact, $(Z,g_t)$ is Ricci parallel
	if and only if, up to quotients, $(M,g)$ is $\SS^4$, $\cp^2$, a 
	flat torus or a K3 surface. 
	\end{theorem}
	
	\section*{Acknowledgements}
	The author 
	would like to kindly thank Professor Paolo Mastrolia (Università degli
	Studi di Milano) and
	Professor Giovanni Catino (Politecnico di Milano) for their helpful suggestions and comments
	about the initial draft of this paper. Also, the author would like
	to express his gratitude to Professor Gérard Besson (Institut Fourier, 
	Université de Grenoble) for the numerous
	conversations concerning
	the topic of this paper and for his useful advice. The author is a member
	of the Gruppo Nazionale per le Strutture Algebriche, Geometriche e loro
	Applicazioni (GNSAGA) of INdAM (Istituto Nazionale di Alta Matematica).
	
	\bibliographystyle{abbrv}
	\bibliography{Seccurvbiblio}
\end{document}